\documentclass[a4paper, 12pt]{amsart}
\usepackage{amssymb,amsmath}
\makeatletter
\@namedef{subjclassname@2010}{%
  \textup{2010} Mathematics Subject Classification}
\makeatother

\newcommand{\A}{\mathbb{A}}

\newcommand{\Spec}{\operatorname{Spec}}

\newcommand{\trdeg}{{\rm tr.deg}\:}
\newtheorem{thm}{Theorem}[section]
\newtheorem{prop}[thm]{Proposition}
\newtheorem{lem}[thm]{Lemma}

\newtheorem{question}[thm]{Question}

\newtheorem{remark}[thm]{Remark}

\begin{document}
\title[Retracts of polynomial rings]{A note on retracts of polynomial rings in three variables}
\author{Takanori Nagamine}
\address[T. Nagamine]{Graduate School of Science and Technology, Niigata University, 8050 Ikarashininocho, Nishi-ku, Niigata 950-2181, Japan}
\email{t.nagamine14@m.sc.niigata-u.ac.jp}
\date{November 24, 2018}
\thanks{Research of the author was partially supported by Grant-in-Aid for JSPS Fellows  (No.\ 18J10420) from Japan Society for the Promotion of Science}

\begin{abstract}
For retracts of the polynomial ring, in \cite{Cos77}, Costa asks us whether every retract of $k^{[n]}$ is also the polynomial ring or not, where $k$ is a field. In this paper, we give an affirmative answer in the case where $k$ is a field of characteristic zero and $n = 3$. 
\end{abstract}
\maketitle

\setcounter{section}{0}

\section{Introduction} 
Let $A$ and $B$ be commutative rings. We say $A$ is a \emph{retract} of $B$ if $A$ is a subring of $B$ and there exists an ideal $I$ of $B$ such that $B \cong A \oplus I$ as $A$-modules. The followings are basic properties of retracts. 
\begin{prop} {\rm (cf.\ \cite[Section 1]{Cos77})} 
\label{prop:Costa}
Let $B$ be an integral domain and let $A$ be a retract of $B$. Then the following assertions hold true. 
	\begin{enumerate}
	  \item
	  $A$ is algebraically closed in $B$. 
	  \item
	  If $B$ is a UFD, then $A$ is also a UFD.
	  \item
	  If $B$ is regular, then $A$ is also regular.  
	\end{enumerate} 
\end{prop}  
\begin{lem} \label{lemma:extension}
Let $k$ be a field. Let $A$ and $B$ be $k$-algebras. If $A$ is a retract of $B$, then $A \otimes_kK$ is a retract of $B \otimes_kK$ for any field $K$ containing $k$. 
\end{lem}
\begin{proof}
Since $A$ is a retract of $B$, there exists an ideal $I$ of $B$ such that $B \cong A \oplus I$ as $A$-modules. Let $K$ be a field containing $k$. Taking a tensor product by $K$ over $k$, we have $B \otimes_kK \cong (A \otimes_kK) \oplus I'$ as $A \otimes_kK$-modules, where $I' := I \otimes_kK$ is an ideal of $B \otimes_kK$. Thus, $A \otimes_kK$ is a retract of $B \otimes_kK$. 
\end{proof}
Let $k$ be a field. We denote $k^{[n]}$ by the polynomial ring in $n$ variables over $k$. In \cite{Cos77}, Costa asks us the following question. 
\begin{question} \label{question} 
Let $k$ be a field and let $B := k^{[n]}$. Then, is every retract of $B$ containing $k$ the polynomial ring?   
\end{question}
If $n \leq 2$, then the above question is affirmative and proved by Costa (\cite[Theorem 3.5]{Cos77}). On the other hand, it is well known that Question \ref{question} is related to Zariski's cancellation problem as below. 
\begin{prop}
If Question \ref{question} holds for $n + 1$, then Zariski's cancellation problem has an affirmative answer for $\A^{n}$, that is, $X \times \A^1 \cong \A^{n + 1}$ implies $X \cong \A^{n}$. 
\end{prop}  
\begin{proof}
Let $k$ be afield. Suppose that Question \ref{question} holds for $n + 1$. Let $X$ be an affine variety over $k$ such that $X \times \A^1_k \cong \A^{n + 1}_k$ and let $A$ be the coordinate ring of $X$. Then $A^{[1]} \cong k^{[n + 1]}$ and $\trdeg_kA = n$. It is clear that $A$ is a retract of $k^{[n + 1]}$. Therefore $A \cong k^{[n]}$, which implies that $X \cong \A^{n}$. 
\end{proof}
When $k$ is a field of positive characteristic, Gupta \cite{Gup14} proved that Zariski's cancellation problem does not hold for $\A^n$ if $n \geq 3$. Therefore Question \ref{question} does not hold in the case where $k$ is a field of positive characteristic and $n \geq 4$. So, the remaining cases are  
	\begin{itemize}
	  \item
	  the characteristic of $k$ is positive and $n = 3$, 
	  \item
	  the characteristic of $k$ is zero and $n \geq 3$. 
	\end{itemize}

In this paper, we consider the case where $k$ is a field of characteristic zero and $n = 3$. The main result in this paper is to give an affirmative answer for Question \ref{question} in this case.
\section{Main results}
Let $k$ be an algebraically closed field and let $X$ be a (not necessarily complete) nonsingular algebraic variety over $k$. By virtue of Nagata's Completion Theorem (\cite{Nag62}), there exists a complete algebraic variety $\overline{X}$ over $k$ such that $X$ is open and dense in $\overline{X}$. We say that $X$ is a \emph{resolvable variety} if Hironaka's Main Theorems about resolution of singularities hold for $\overline{X}$ and $\overline{X} - X$. For a resolvable variety $X$, we denote $\bar{\kappa} (X)$ by the \emph{logarithmic Kodaira dimension}. 
\begin{lem} {\rm (cf.\ \cite[Theorem 1.1 (a)]{Kam80})} 
\label{lem:1} 
Let $f: X \to Y$ be a morphism of nonsingular, resolvable algebraic varieties over an algebraically closed field. If $f$ is dominant and generically separable, then $\bar{\kappa} (X) \geq \bar{\kappa} (Y)$.  
\end{lem}
The following is a characterization for affine planes (see Miyanishi \cite{Miy75}, Fujita \cite{Fuj79}, Miyanishi--Sugie \cite{MS80} and Russell \cite{Rus81}).  
\begin{thm} \label{thm:MSR} 
Let $k$ be an algebraically closed field and $X$ be a nonsingular affine surface over $k$. Let $A$ be the coordinate ring of $X$, namely $X = \Spec A$. Then $X \cong \A^2_k$ if and only if $A^* = k^*$, $A$ is a UFD and $\bar{\kappa} (X) = - \infty$. 
\end{thm} 
First of all, we shall show some properties of retracts of the polynomial ring over a field. 
\begin{lem}
\label{keylemma}
Let $k$ be an algebraically closed field and let $B := k^{[n]}$ be the polynomial ring in $n$ variables over $k$. Let $A$ be a retract of $B$ containing $k$ and set $X = \Spec A$. If $X$ is resolvable and $Q(B)$ is separably generated over $Q(A)$, then the following assertions hold true.
	\begin{enumerate}
	  \item
	  $A$ is a finitely generated UFD over $k$ with $A^* = k^*$,  
	  \item
	  $X$ is a nonsingular variety over $k$,  
	  \item
	  $\bar{\kappa} (X) = - \infty$,  
	\end{enumerate}
where we denote $Q(R)$ by the quotient field of an integral domain $R$. 
\end{lem}
\begin{proof} 
Since $A$ is a retract of $B$, $A$ is a $k$-subalgebra of $B$. Hence it is clear that $A$ is an integral domain with $A^* = k^*$. Furthermore, there exists a surjective homomorphism as $k$-algebras $\varphi : B \to A$ such that the following sequence of $A$-modules is exact and split: 
	\begin{center}
	  $0 \xrightarrow{} I \xrightarrow{} B \xrightarrow{\varphi}  A \xrightarrow{} 0$, 
	\end{center} 
where $I := \ker f$. Hence $A$ is finitely generated as a $k$-algebra. Also by Proposition \ref{prop:Costa} (2), we see that $A$ is a UFD.  

We consider a morphism $f : \A^n_k \cong  \Spec B \to X$ defined by a natural inclusion map $\iota : A \to B$. It follows from Proposition \ref{prop:Costa} (3) that $X$ is nonsingular over $k$. 

Suppose that $X$ is resolvable and $Q(B)$ is separably generated over $Q(A)$. Then $f$ is dominant and generically separable. Hence by Lemma \ref{lem:1}, we have $\bar{\kappa} (X) \leq \bar{\kappa} (\A^n_k) =  - \infty$, which implies that $\bar{\kappa} (X) =  - \infty$. 
\end{proof}

When we consider the polynomial ring in two variables whose ground field is not necessarily algebraically closed, the following result is useful. 
\begin{thm} {\rm (cf.\ \cite{Kam75} or \cite[Theorem 5.2]{Fre17})}
\label{thm:Kambayashi}
Let $K$ and $k$ be fields such that $K$ is separably generated over $k$. Suppose that $A$ is a commutative $k$-algebra for which $K\otimes_kA \cong K^{[2]}$. Then $A \cong k^{[2]}$ . 
\end{thm}

The following is the main result in this paper. 
\begin{thm} \label{thm:main}
Let $k$ be a field of characteristic zero and let $B := k^{[3]}$ be the polynomial ring in three variables over $k$. Then every retract of $B$ is isomorphic to the polynomial ring.  
\end{thm}
\begin{proof}
Let $A$ is a retract of $B$ and let $d := \trdeg_k(A)$. Clearly, if $d = 0$, then $A = k$. By Proposition \ref{prop:Costa}, $A$ is algebraically closed in $B$. Hence, if $d = 3$, then $A = B = k^{[3]}$. If $d = 1$, then we already know that $A \cong k^{[1]}$ by \cite[Theorem 3.5]{Cos77}.

Suppose that $d = 2$. Let $K$ be an algebraic closure of $k$. Set $A_K := A \otimes_kK$ and $B_K := B \otimes_kK$. It follows from Lemma \ref{lemma:extension} that $A_K$ is also retract of $B_K = K^{[3]}$. Set $X = \Spec A_K$. By using Lemma \ref{keylemma}, we have $X$ is a nonsingular, factorial surface over $K$ and $(A_K)^* = K^*$. Therefore it follows from Theorem \ref{thm:MSR} that $A_K \cong K^{[2]}$. Applying Theorem \ref{thm:Kambayashi} for $A_K$, we have $A \cong k^{[2]}$.  
\end{proof}

\begin{remark}
{\rm 
In Lemma \ref{keylemma}, we don't know whether $Q(B)$ is separably generated over $Q(A)$ or not in general. Of course, if it is true in general, then Theorem \ref{thm:main} holds true for any characteristic. 
}
\end{remark}

\section*{Acknowledgments}
The author is sincerely grateful to Professor Dayan Liu for letting me know Costa's question.



\begin{thebibliography}{30}
\bibitem{Cos77}
D. Costa, Retracts of polynomial rings, J. Algebra {\bf 44} (1977), 492--502. 
\bibitem{Fre17}
G. Freudenburg, Algebraic Theory of Locally Nilpotent Derivations (second edition), Encyclopedia of Mathematical Sciences vol.\ 136, Invariant Theory and Algebraic Transformation Groups VII, Springer-Verlag, 2017. 
\bibitem{Fuj79}
T. Fujita, On Zariski problem, Proc.\ Japan Acad., Ser.\ A, {\bf 55} (1979), 106--110. 
\bibitem{Gup14}
N. Gupta, On Zariski's cancellation problem in positive characteristic, Adv.\ Math., {\bf 264} (2014), 296--307. 
\bibitem{Kam75}
T. Kambayashi, On the absence of non-trivial separable forms of the affine plane, J. Algebra {\bf 35} (1975), 449--456 . 
\bibitem{Kam80}
T. Kambayashi, On Fujita's strong cancellation theorem for the affine plane, J. Fac.\ Sci.\ Univ.\ Tokyo Sect.\ IA Math., {\bf 27} (1980), no.\:3, 535--548. 
\bibitem{Miy75}
M. Miyanishi, An algebraic characterization of the affine plane, J. Math.\ Kyoto Univ., {\bf 15} (1975), 169--184. 
\bibitem{MS80}
M. Miyanishi and T. Sugie, Affine surfaces containing cylinderlike open sets, J. Math.\ Kyoto Univ., {\bf 20} (1980), 11--42. 
\bibitem{Nag62}
M. Nagata, Imbedding of an abstract variety in a complete variety, J. Math.\ Kyoto Univ., Ser.\:A 2 (1962), 1--10. 
\bibitem{Rus81}
P. Russell, On affine ruled rational surfaces, Math.\ Ann., {\bf 255} (1981), 287--302.
\end{thebibliography}
\end{document}